\documentclass[final,a4paper]{article}
\usepackage[french,english]{babel}
\usepackage{amsmath, amsfonts,amssymb,amsopn,amscd,amsthm,mathrsfs}
\usepackage{bm}
\usepackage{graphicx}
 \usepackage{enumerate}
 \usepackage{array}
 \usepackage{hyperref}
 \hypersetup{colorlinks=true,linkcolor=blue,citecolor=blue,linktoc=page}
\usepackage{showkeys}
\numberwithin{equation}{section}
\usepackage{fancyhdr}
\newcommand{\eq}{\begin{equation}}
\newcommand{\qe}{\end{equation}}

\newcommand{\E}{\mathbb{E}}

\newcommand{\R}{\mathbb{R}}
\newcommand{\p}{\mathbb{P}}
\theoremstyle{plain}

\newtheorem{thm}{Theorem}[section]
\newtheorem*{thm*}{Theorem}
\newtheorem{lem}{Lemma}[section]
\newtheorem{prop}{Proposition}[section]

\theoremstyle{definition}

\theoremstyle{remark}
\newtheorem*{rem}{Remark}

\usepackage{color}

\begin{document}
\sloppy
\pagestyle{headings} 
\title{Improved one-sided deviation inequalities under regularity assumptions for product measures}
\date{Note of \today}
\author{Kevin Tanguy \\ University of Angers, France}

\maketitle

\begin{abstract}
This note is concerned with lower tail estimates for product measures. Some improved deviation inequalities are obtained for functions satisfying some regularity and monotonicity assumptions. The arguments are based on semigroup interpolation together with Harris's negative association inequality and hypercontractive estimates.
\end{abstract}


\section{Introduction}
As an introduction we recall some facts about Gaussian concentration of measure (cf. \cite{Led}) and Superconcentration theory (cf. \cite{Chatt1}).\\

It is well known that concentration of measure is an effective tool in various mathematical areas (cf. \cite{BLM}).
 In a Gaussian setting, classical concentration results typically state that, for a Lipschitz function $f\,:\,\R^n\to \R$  with Lipschitz constant $\|f\|_{{\rm Lip}}$, 

\begin{equation}\label{eq.sudakov.tsirelson}
\gamma_n\big(|f-\E_{\gamma_n}[f]|\geq t\big)\leq 2e^{-\frac{t^2}{2\|f\|_{Lip}^2}}, \quad t\geq 0,
\end{equation}

\noindent with $\gamma_n$ the standard Gaussian measure on $\R^n$. Another example of concentration of measure is the Poincar\'e inequality satisfied by $\gamma_n$. Namely, for $f\in L^2(\gamma_n)$ smooth enough : 

\begin{equation}\label{eq.poincare.gaussien}
{\rm Var}_{\gamma_n}(f)\leq \int_{\R^n}|\nabla f|^2d\gamma_n,
\end{equation}

\noindent where $|\cdot|$ stands for the Euclidean norm on $\R^n$. As effective as \eqref{eq.sudakov.tsirelson} and \eqref{eq.poincare.gaussien} are, their generality can lead to sub-optimal bounds in some particular cases. For instance, consider the $1$-Lipschitz function on $\R^n$ $f(x)=\max_{i=1,\ldots,n}x_i$. At the level of the variance, \eqref{eq.poincare.gaussien} gives 

\[
{\rm Var}(M_n)\leq 1, 
\]

\noindent with $M_n=\max_{i=1,\ldots,n}X_i$ where $(X_1,\ldots,X_n)$ stands for a standard Gaussian random vector in $\R^n$, whereas it has been proven that ${\rm Var}(M_n)\leq C/\log n$ with $C>0$ a numerical constant. At an exponential level \eqref{eq.sudakov.tsirelson} is not satisfying either. Indeed, it is well known in Extreme Value theory (cf. \cite{Lead}, pages $14-15$ ) that $M_n$ can be renormalized by some numerical constants   $a_n=\sqrt{2\log n}$ and $b_n=a_n-\frac{\log 4\pi+\log\log n}{2a_n}$, $n\geq 1$, such that

\[
a_n(M_n-b_n)\to \Lambda_0\quad \text{as}\quad  n\to\infty
\]

\noindent in distribution where $\Lambda_0$  corresponds to the Gumbel distribution with cumulative distribution function :

\[
\p(\Lambda_0\leq x)=\exp(-e^{-x}), \quad x\in \R.
\]

\noindent Then, it is clear that the asymptotics of $\Lambda_0$ are not Gaussian but rather exponential on the right tail and double exponential on the left tail. It is now obvious that \eqref{eq.sudakov.tsirelson} and \eqref{eq.poincare.gaussien} lead to sub-optimal results for the function $f(x)=\max_{i=1,\ldots,n}x_i$. This is referred to as Superconcentration phenomenon (cf. \cite{Chatt1}). This kind of phenomenon occurs for different functionals of Gaussian random variables and has been studied in \cite{BT, KT, KT2, Paou, Val}\ldots.\\

Recently, additional convexity assumption has been fruitfully used by Paouris and Valettas in order to improve the concentration inequality \eqref{eq.sudakov.tsirelson}. In the context of small ball probabilities and random Dvoretzky's Theorem, these two authors improved the lower tail of any convex function thanks to Ehrard's inequality in \cite{PaouVal}.  More precisely, they obtained 

\begin{thm}\label{thm.paouris.valettas}[Paouris,Valettas]
Let $f\,:\, \R^n\to \R$ be a convex function, then the following holds

\begin{equation}\label{eq.paouris.valettas}
\gamma_n\bigg(f-\int_{\R^n}fd\gamma_n\leq -t\bigg)\leq e^{-c\frac{t^2}{{\rm Var}_{\gamma_n}(f)}},\quad t>1
\end{equation}
\noindent where $c>0$ is a universal constant.
\end{thm}
\begin{rem}
Of course, the improvements stays in the fact that ${\rm Var}_{\gamma_n}(f)\leq \|f\|_{{\rm Lip}}^2$ as we have just seen on the basic example of the maximum of $n$ independent standard Gaussian random variables. Ehrhard's inequality has also been used by Valettas in \cite{Val} where he proved that \eqref{eq.sudakov.tsirelson} is tight if the convex function $f$ is not superconcentrated. \\

Besides, the work from \cite{PaouVal} has been used by Valettas to extend Theorem \ref{thm.paouris.valettas}. Indeed,  as consequence of his inequality with Paouris, combined with transportation-type arguments, he obtained (cf. \cite{Val}, section $2.1.3$) concentration inequalities for nondecreasing, convex functions  in a log-concave measures setting. 
\end{rem}

The purpose of this note is the following : semigroup's arguments together with Harris negative association Lemma and hypercontractive estimates will be used to obtain a deviation inequality for the lower tail of functions belonging to $\mathcal{F}_+$ where 

\[
\mathcal{F}_+=\{ f\in C^2(\R^n,\R)\quad ;\quad \text{monotone with}\quad \partial_{ij}^2f\geq 0\,\, \forall i,j=1,\ldots,n\}.
\]
The obtained deviation inequalities will be similar to the deviation from Theorem \ref{thm.paouris.valettas} (and its extension). However the class of measures will be different (not necessarily larger) and the proof will be based on interpolation by semigroups arguments. \\

Now, let us describe in more details our setting and state our main result. \\

Let $n\geq 1$ be fixed and consider $\mu=\mu_1\otimes\ldots\otimes\mu_n$ where, for any $i=1,\ldots,n$,  $d\mu_i=e^{-V_i(x)}dx$ are probability measures on $\mathcal{B}(\R)$, the Borel $\sigma$-algebra of $\R$, and $V_i\,:\,\R\to\R$ are smooth potentials. In the sequel, we will assume that there exists $\kappa_i\in \R$ such that

\[
V_i''(x)\geq -\kappa_i\quad \forall x\in\R\quad \text{and}\quad i=1,\ldots,n
\]
and will denote by $\kappa=\max_{i=1,\ldots,n}\kappa_i$. \\

Now, let us recall some facts about functional inequalities and their links with related semigroups. General references on semigroups, functional inequalities and concentration of measures are \cite{BGL,Led, BLM}.\\

In our setting, $d\mu(x)=e^{-V(x)}dx$ is a probability measure on $\mathcal{B}(\R^n)$, the Borel $\sigma$-algebra of $\R^n$, with
\[
 V(x)=\sum_{i=1}^nV_i(x_i)\quad \text{and} \quad x=(x_1,\ldots,x_n)\in\R^n.
 \]
  It is classical that such measures can be seen as an invariant and reversible measure of the associated diffusion operator $L=\Delta-\nabla V\cdot \nabla$. The operator $L$ generates the Markov semigroup of operators $(P_t)_{t\geq 0}$ and defines by integration by parts the Dirichlet form

\begin{equation}\label{eq.ipp}
\mathcal{E}(f,g)=\int_{\R^n}f(-Lg)d\mu=\int_{\R^n}\nabla f\cdot\nabla gd\mu
\end{equation}

\noindent for some smooth functions $f,g$ on $\R^n$. The set of functions for which the preceding expression make sense is called the Dirichlet domain of $L$. We design by $\mathcal{D}(L)$ such set.\\

Given such a couple $(L,\mu)$, it is said to satisfy a spectral gap, or Poincar\'e, inequality if there is a constant $\lambda>0$ such that for all functions $f$ of the Dirichlet domain

\begin{equation}\label{eq.poincare}
\lambda {\rm Var}_{\mu}(f)\leq\mathcal{E}(f,f).
\end{equation}

\noindent with ${\rm Var}_\mu(f)=\int_{\R^n}f^2d\mu-(\int_{\R^n}fd\mu)^2$. Similarly, it satisfies a logarithmic Sobolev inequality if there exists a constant $\rho>0$ such that for all functions $f$ of the Dirichlet domain, 

\begin{equation}\label{eq.logsob}
\rho{\rm Ent}_{\mu}(f^2)\leq 2\mathcal{E}(f,f).
\end{equation}

\noindent with ${\rm Ent}_\mu(f)=\int_{\R^n}f\log fd\mu-(\int_{\R^n}fd\mu)(\log\int_{\R^n}fd\mu)$ and $f>0$.\\

 One speaks of the spectral gap constant (of $(L,\mu)$) as the largest $\lambda>0$ for which \eqref{eq.poincare} holds, and of the logarithmic Sobolev constant (of $(L,\mu)$) as the best $\rho>0$ for which \eqref{eq.logsob} holds. We still use $\lambda$ and $\rho$ to design these constants. It is classical (cf.  \cite{Led}) that $\rho\leq \lambda$.\\

A particular feature of the logarithmic Sobolev inequalities is the (equivalent, cf. \cite{Gro}) hypercontractive property of the semigroup. Precisely, the logarithmic Sobolev inequality \eqref{eq.logsob} is equivalent to saying that, whenever $p\geq 1+e^{2\rho t}$, for all functions $f$ in $L^p(\mu)$, 

\begin{equation}\label{eq.hc}
\|P_t(f)\|_2\leq \|f\|_p
\end{equation}

For simplicity, we say below that a probability measure $\mu$, in this context, is hypercontractive with constant $\rho$.\\

 Finally, let us also recall that an Orlicz norm $\|\cdot\|_{\phi}$ is defined as follow : given a Young function $\phi$, set

\[
\|f\|_{\phi}=\inf\bigg\{c>0\,;\, \int_{\R^n}\phi\bigg(\frac{|f|}{c}\bigg)d\mu\geq 1\bigg\}
\]

\noindent the associated Orlicz norm of a measurable function $f\,:\, \R^n\to \R$. In the sequel, let $\phi\,:\,\R+\to\R_+$ be convex such that $\phi(x)=\frac{x^2}{\log (e+x)}$ for $x\geq 1$ and $\phi(0)=0$. To ease the notation, we set  $\|\nabla f\|_\phi^2$ as a shorthand for $\sum_{i=1}^n\|\partial_i f\|_\phi^2$ with $\partial_i$, for any $i\in\{1,\ldots,n\}$, stands for the $i$-th partial derivative operator.\\

 In this context, the following Theorem is our main result.

\begin{thm}\label{thm2}
Within the preceding framework, assume that $(\mu_i)_{i=1,\ldots,n}$ are hypercontractive with constant $\rho$. Then, for any smooth  $f\in\mathcal{F}_+$ we have 

\begin{equation}\label{eq.thm2}
{\rm Ent}_{\mu}(e^{-f})\leq C_{\rho,\kappa}\|\nabla f\|_{\phi}^2 \E_{\mu}[e^{-f}].
\end{equation}

\noindent where $C_{\rho,\kappa}=\frac{2e^{[1+(\kappa/\rho)]_+}}{\rho(1-e^{-2\rho T})}$ for some $T>0$. In particular, the following holds 

\begin{equation}\label{eq.deviation1}
\mu\bigg(f-\int_{\R^n}fd\mu\leq -t\bigg)\leq  e^{-c_{\rho,\lambda}\frac{t^2}{2\|\nabla f\|_\phi^2}},\quad t\geq 0
\end{equation}
\noindent where $c_{\rho,\lambda}>0$ is a universal constant.
\end{thm}

\begin{rem}
\begin{enumerate}
\item In practice, it is classical to bound (cf. \cite{CoLed} )$\|\nabla f\|_{\phi}^2$ by the following quantity :

\[
\|\nabla f\|_{\phi}^2=\sum_{i=1}^n\|\partial_if\|_{\phi}^2\leq C\sum_{i=1}^n\frac{\|\partial_i f\|_2^2}{1+\log \|\partial_i f\|_1/\|\partial_i f\|_2}
\]
with $C>0$ a numerical constant.
\item When, the standard Gaussian measure is considered 
\[
i.e. \quad V_i(x)=\frac{x^2}{2},\quad i=1,\ldots,n\quad \text{and}\quad x\in\R
\]
 the quantity $\|\nabla f\|_{\phi}^2$ can be replaced by the variance ${\rm Var}_{\gamma_n}(f)$ which is smaller.
\end{enumerate}
\end{rem}

We want to highlight the fact that only $\kappa\in \R$ is required here, it appears as a mild property shared by numerous potentials such as, for example, double-wells potentials on the line of the form $V(x)=ax^4-bx^2,\, a,b>0$. The stronger strict convexity assumption $V''\geq \rho>0$ (satisfied by the standard Gaussian measure $\gamma_n$) actually implies that $\mu$ satisfies a logarithmic Sobolev inequality, and thus hypercontractivity, with constant $\rho$ (cf. \cite{BGL}).\\

To better understand where the improvement lies in Theorem \ref{thm2} Let us recall some facts : for a smooth function $f\,:\,\R^n\to\R$ it is known (cf. the introduction of \cite{Val} and references therein) that

\[
{\rm Var}_\mu(f)\leq \|\nabla f\|_{\phi}^2\leq \E_\mu[|\nabla f|^2]\leq \|f\|_{{\rm Lip}}^2
\]

\noindent and each terms can be different from one another.  For instance (cf. \cite{KT1,BT}), in a Gaussian case, if $f(x)=\max_{i=1,\ldots,n}x_i$ 

\[
{\rm Var}_{\gamma_n}(f)\simeq \frac{1}{\log n}\simeq \|\nabla f\|_\phi^2\quad \text{and} \quad \E_{\gamma_n}[|\nabla f|^2]=\|f\|_{{\rm Lip}}^2=1.
\]

If $f(x)={\rm Med}(x_1,\ldots,x_n)$, we have 

\[
{\rm Var}_{\gamma_n}(f)\simeq \frac{1}{n},\quad \|\nabla f\|_\phi^2\simeq\frac{1}{\log n}\quad \text{and} \quad \E_{\gamma_n}[|\nabla f|^2]=\|f\|_{{\rm Lip}}^2=1.
\]

Let us mention that  \eqref{eq.sudakov.tsirelson} has already been improved for convex functions, with $\E_{\gamma_n}[|\nabla f|^2]$ instead of $\|f\|_{{\rm Lip}}^2$, by Paouris and Valettas (cf. \cite{PaouVal3} section $5.2$). Thus, in Theorem \ref{thm2}, we obtain something slightly better. However, this bound
is a priori larger (except for the Gaussian case) than the one involving ${\rm Var}_\mu(f)$ which would be the desired one for every $\mu$.\\

Now, let us describe the organization of the article. Section \ref{2} is concerned with semigroup facts and negative association. In section \ref{4} we prove Theorem \ref{thm2}. Section \ref{4} will describe some potential extensions. Finally, in section \ref{5}, we say a few words about a recent result from \cite{Dung}. \\

In the sequel, we will always assume that the functions are sufficiently integrable with respect to $\mu$ in order that studied inequalities make sense and the commutation between integrals and derivatives are legit. Also, by convention, $C>0$ is a numerical constant that may change at each occurence.

\section{Tools}\label{2}

\subsection{Semigroup properties}
In this section, we present the tools needed to prove Theorem \ref{thm2}. In the context described in the introduction, let us collect some important properties of the semigroup $(P_t)_{t\geq 0}$. Again, for more details,  the reader is referred to \cite{BGL} (or \cite{Led2}, pages $306-328$, for a shorter exposition).

\begin{prop}
Within the preceding framework, the following holds 
\begin{enumerate}
\item [$\bullet$] For any smooth function $f\,:\R^n\to\R$, the semigroup $(P_t)_{t\geq 0}$ solves the heat equation associated to $L$. 

\begin{equation}\label{eq.heat.equation}
 i.e. \quad \partial_tP_t(f)=LP_t(f)=P_t(Lf)\quad  \text{for any}\quad  t\geq 0.
 \end{equation}
\item [$\bullet$] $(P_t)_{t\geq 0}$ is ergodic : for any smooth function $f\,:\,\R^n\to\R$
\begin{equation}\label{eq.ergodicity}
\lim_{t\to+\infty}P_t(f)=\E_\mu[f]
\end{equation}
\item [$\bullet$] For any $i=1,\ldots,n$ and any smooth function $f\,:\,\R^n\to\R$, the uniform lower bound $V''_i\geq -\kappa_i$, is equivalent to the following commutation property
\begin{equation}\label{eq.commutation}
|\partial_iP_t(f)|\leq e^{\kappa t}P_t(|\partial_if|)\quad \text{for any}\quad  t\geq 0
\end{equation}
\end{enumerate}
\end{prop}
 \begin{rem}
 When $\mu=\gamma_n$ the commutation property \eqref{eq.commutation} is exact (cf. \cite{Led2, BGL}) Namely, for any $i=1,\ldots,n$ and any smooth function $f\,:\,\R^n\to\R$
 
 \begin{equation}\label{eq.ornstein.commutation}
 \partial_iP_t(f)=e^{-t}P_t(\partial_if)\quad \text{for any}\quad  t\geq 0.
 \end{equation}
 This fact can also be checked on the representation formula \eqref{eq.mehler} given in the sequel.
 \end{rem}

\subsection{Semigroup representation of the Entropy}
As it will be needed in the sequel, we state below some representation (cf. \cite{BGL} section $5.5$ or section $2.1$ in \cite{Led2}) of the entropy of a function along the semigroup $(P_t)_{t\geq 0}$. \\

\begin{equation}\label{eq.representation.entropy1}
{\rm Ent}_{\mu}(f^2)=\int_0^{+\infty}\int_{\R^n} \frac{|\nabla P_t (f^2)|}{P_t (f^2)}d\mu dt
\end{equation}

As it is exposed in \cite{CoLed}, when $\mu$ satisfies a logarithmic Sobolev inequalities there is no need to deal with large value of $t$ in \eqref{eq.representation.entropy1}. Indeed, a logarithmic Sobolev inequalities is equivalently stated as a exponential decay of the entropy along the semigroup. Namely, 

\begin{equation}\label{eq.entropy.decay}
{\rm Ent}_\mu\big(P_t(f)\big)\leq e^{-2t\rho}{\rm Ent}_\mu(f)\quad \text{for every}\quad t\geq 0
\end{equation}
\noindent and every positive function $f$ in $L^1(\mu)$. Therefore, the combination of the preceding representation \eqref{eq.representation.entropy1} by semigroup together with the exponential decay of the the entropy (cf. \cite{BGL} page $244$) along the semigroup we have, for any $T> 0$,

\begin{equation}\label{eq.representation.entropy2}
{\rm Ent}_{\mu}(f^2)\leq \frac{1}{1-e^{-2\rho T}}\int_0^{T}\int_{\R^n} \frac{|\nabla P_t (f^2)|}{P_t (f^2)}d\mu dt
\end{equation}
In the sequel, we choose e.g. $T=\frac{1}{2\rho}$.
 \subsection{Semigroup and Harris inequality}
 
 As mentioned earlier, in order to investigate the lower tail, one has to use negative association inequalities. Therefore we state below Harris's Lemma (cf.\cite{BLM} page $43$) and see how it can be combined with semigroups. Recall that monotonicity or convexity properties of a function $f\,:\, \R^n\to \R$ are understood coordinate-wise.

\begin{prop}[Harris's negative association inequality]\label{prop.harris}
Let $f\,:\,\R^n\to\R$ and $g\,:\,\R^n\to\R$  two monotone functions with different monotonicity, then 

\begin{equation}\label{eq.harris}
\E\big[f(X)g(X)\big]\leq \E\big[f(X)\big]\E\big[g(X)\big]\quad \text{for} \quad X=(X_1,\ldots,X_n)
\end{equation}

with $X_i$ independent random variables.
\end{prop}



In the sequel,  this proposition will also be used at the level of the semigroup. That is to say for the underlying heat kernel measure $p_t(x,dy)$ which is defined (cf. \cite{BGL} page $12$) as
\[
P_t(f)(x)=\int_{\R^n}f(y)p_t(x,dy)\quad  \text{with}\quad t\geq 0\quad \text{and}\quad  x\in\R^n
\]

This is the content of the following Lemma.
\begin{lem}\label{lem.harris.semigroup}
 Let $t\geq 0$ and $x\in\R^n$ be fixed and consider $f$ and $g$ two monotone functions with different monotonicity, then

\[
P_t(fg)(x)\leq P_t(f)(x)P_t(g)(x)
\]
\end{lem}



The following Lemma explains, in our context, that the semigroup $(P_t)_{t\geq 0}$ preserves monotonicity properties of a function.

 \begin{lem}\label{lem.intertwinnings}
Let $f\,:\,\R^n\to \R$ be monotone, then $x\mapsto P_t(f)(x), \, t\geq 0$ shares the same monotonicity properties as the function $f$.
\end{lem}
\begin{proof}
As it is exposed in \cite{MalTal}, in our setting, we have the following representation of $\nabla P_tf(x)$ for any $x\in\R^n$ and $t\geq 0$.

\begin{equation}\label{eq.intertwinnings}
\nabla P_tf(x)=\E\big[\nabla f(X_t)e^{-\int_0^t V''(X_s)ds}\big| X_0=x\big]
\end{equation}

\noindent Thus, $x\mapsto P_tf(x)$ shares the same monotonicity properties as $f$.
\end{proof}
\begin{rem}
\begin{enumerate}
\item In the Gaussian setting, for quadratic potentials, this property is obvious thanks to Mehler's formula which gives an explicit representation of the Ornstein-Uhlenbeck semigroup : 

\begin{equation}\label{eq.mehler}
P_tf(x)=\int_{\R^n}f(xe^{-t}+\sqrt{1-e^{-2t}}y)d\gamma_n(y),\quad t\geq 0,\, x\in\R^n
\end{equation}

\item Representation as \eqref{eq.intertwinnings} is part of the so-called intertwinnings relation between a semigroup with some differential operator (cf.\cite{JouBon1,JouBon2} and references therein).\\
\item The fact that a semigroup preserves the monotonicity of a function has also been investigate in \cite{Page}.
\end{enumerate}
\end{rem}

\section{Study of the lower tail - Proof of Theorem \ref{thm2}}\label{3}

Recall that the measures $(\mu_i)_{i=1,\ldots,n}$ are assumed to be hypercontractive with constant $\rho$. In this section we prove Theorem \ref{thm2} thanks to Lemma \ref{lem.harris.semigroup} and \ref{lem.intertwinnings}.

\begin{proof}
Let $f\in\mathcal{F}_+$ be. Without loss of generality, we can assume that $f$ is non-decreasing : i.e. $\partial_if\geq 0$ for all $i\in\{1,\ldots,n\}$. Then, start with the representation formula \eqref{eq.representation.entropy2}

\begin{equation*}
{\rm Ent}_{\mu}(f^2)\leq \frac{1}{1-e^{-2\rho T}}\int_0^{T}\int_{\R^n} \frac{|\nabla P_t (f^2)|}{P_t (f^2)}d\mu dt
\end{equation*}

\noindent and apply it to $e^{-f}$. We obtain, thanks to the commutation properties \eqref{eq.commutation},

\begin{eqnarray*}
{\rm Ent}_{\mu}(e^{-f})&\leq&\frac{1}{1-e^{-2\rho T}}\int_0^T e^{2\kappa t}\sum_{i=1}^n\int_{\R^n}\frac{P_t^2(\partial_ife^{- f})}{P_t (e^{- f})}d\mu dt
\end{eqnarray*}
Notice that, for any $\in\{1,\ldots,n\}$, $\partial_if$ and $e^{-f}$ are monotone with different monotonicity. Therefore, by Lemma \ref{lem.intertwinnings}, this is also the case for $P_t(\partial_if)$ and $P_t(e^{-f})$. Then, by applying Lemma \ref{lem.harris.semigroup} twice, we get

\begin{eqnarray*}
{\rm Ent}_{\mu}(e^{-\theta f})&\leq &\frac{1}{1-e^{-2\rho T}}\int_0^T e^{2\kappa t}\sum_{i=1}^n\int_{\R^n}P_t^2(\partial_if)\frac{P_t^2(e^{- f})}{P_t (e^{- f})}d\mu dt\\
&\leq &\frac{1}{1-e^{-2\rho T}}\E_{\mu}[e^{- f}]\times\int_0^T e^{2\kappa t}\sum_{i=1}^n\int_{\R^n}P_t^2(\partial_if)d\mu dt
\end{eqnarray*}
where in the last upper bound we used that $\mu$ is the invariant measure of $(P_t)_{t\geq 0}$. Namely, $\E_{\mu}[P_t(h)]=\E_\mu[h]$ for any smooth functions $h\,:\,\R^n\to\R$.\\

Finally, in the preceding inequality, the last factor can be upper bounded by hypercontractive arguments. To this task, we follow the proof of Talagrand's inequalities exposed in \cite{CoLed} (pages 8-9) in order to obtain
\begin{equation}\label{eq.hc1}
\int_0^T e^{2\kappa t}\sum_{i=1}^n\int_{\R^n}P_t^2(\partial_if)d\mu\leq \frac{2e^{[1+(\kappa/\rho)]_+}}{\rho}\sum_{i=1}^n\|\partial_if\|_{\phi}^2
\end{equation}

 \noindent To sum up, we have proven

\[
{\rm Ent}_{\mu}(e^{-f})\leq C_{\rho,\kappa}\|\nabla f\|_\phi^2\E_{\mu}[e^{-f}].
\]
with $C_{\rho,\kappa}=\frac{2e^{[1+(\kappa/\rho)]_+}}{\rho(1-e^{-2\rho T})}$. The deviation inequality is classically obtained by applying the preceding inequality to $e^{-\theta f}$ with $\theta\geq 0$.
\end{proof}
\begin{rem}
Let us notice that the preceding scheme of proof can also be done at the level of the variance with the dynamical representation (used in \cite{CoLed})

\[
{\rm Var}_{\mu}(f)=2\int_0^{\infty}\int_{\R^n} |\nabla P_t (f)|^2d\mu dt.
\]
Furthermore, when $\mu=\gamma_n$, one can choose $T=+\infty$. Then, thanks to the exact commutation property \eqref{eq.ornstein.commutation} between $\nabla$ and $(P_t)_{t\geq 0}$ together with the preceding dynamical representation of the variance, we get 

\[
{\rm Ent}_{\gamma_n}(e^{-f})\leq \E_{\gamma_n}[e^{-f}]{\rm Var}_{\gamma_n}(f)
\]
\end{rem}

\section{Potential extensions}\label{4}

Let us say a few words about some potential extensions. As it was emphasized in \cite{CoLed}, one key features  of the preceding methodology is the following. Given a Markov semigroup $(P_t)_{t\geq 0}$ with generator $L$ and invariant measure $\mu$. Assume that $(L,\mu)$ is hypercontractive and that the associated Dirichlet form $\mathcal{E}$ may be decomposed along directions $\Gamma_i$ acting on functions on some state space $E$ as 

\[
\mathcal{E}(f,f)=\sum_{i=1}^n\int_{E}\Gamma_i^2(f)d\mu
\]

\noindent in a way that, for each $i=1,\ldots,n$, $\Gamma_i$ commutes to $(P_t)_{t\geq 0}$ in the sense that, for some constant $\kappa\in \R$, every $t\geq 0$ and $f$ smooth enough, 

\begin{equation}\label{eq.commutation2}
\Gamma_i(P_t f)\leq e^{\kappa t}P_t\big(\Gamma_i (f)\big).
\end{equation}

In the current article, this commutation property is obtained as a strong gradient bound from Bakry and Emery's Gamma 2 criterion and is stated in \eqref{eq.commutation}. \\

As a first example, one can investigate the standard exponential measure (or gamma measure) $d\mu=e^{-\sum_{i=1}^nx_i}1_{\{x_1\geq 0\}}\ldots1_{\{x_n\geq 0\}}dx_1\ldots dx_n$ on $\R^n_+$ with the direction $\Gamma_i(f)=\sqrt{x_i}\partial_i$. According to \cite{BGL, KT}, the commutation properties \eqref{eq.commutation2} is satisfied with $\kappa=-1$. Now, observe that the operator $\Gamma_i,\,i=1,\ldots,n$ preserves the key features of the function $f$. More precisely, assume $f\in\mathcal{F}_+$, then it is easy to check that $x_i\mapsto \Gamma_i(f)\in\mathcal{F}_+$. Besides the following identity, for any $\theta\in\R$, holds
\[
 \Gamma_i(e^{\theta f})=\theta e^{\theta f}\Gamma_i(f).
 \]
  \noindent Therefore, it is possible to apply Harris's negative association \ref{prop.harris} in this situation.\\

Indeed, in this setting, it is then easy to extend slightly the result of the current article. Following the lines of the proof of our main result, we obtain 
\begin{eqnarray*}
{\rm Ent}_{\mu}(e^{-f})&\leq&\frac{1}{2-e^{-2\lambda T}}\int_0^Te^{2\kappa t}\sum_{i=1}^n\int_{\R^n}P_t^2\big(\Gamma_i(f)e^{- f}\big)d\mu dt\\
&\leq &\E_{\mu}[e^{- f}]\times\bigg(\frac{1}{2-e^{-2\lambda T}}\int_0^Te^{2\kappa t}\sum_{i=1}^n\int_{\R^n}P_t^2(\Gamma_i(f))d\mu dt\bigg)\\
&\leq &\E_\mu[e^{- f}]\times C_{\rho,\kappa}\|\Gamma(f)\|_{\phi}^2
\end{eqnarray*}

\noindent where $\|\Gamma (f)\|_{\phi}^2$ is a shorthand for $\sum_{i=1}^n\|\Gamma_i(f)\|_{\phi}^2$. Notice also, according to \cite{CoLed}, that hypercontractive estimates also yields the following upper bound

\[
\|\Gamma (f)\|_{\phi}^2\leq C\sum_{i=1}^n\frac{\|\Gamma_i f\|_2^2}{1+\log\big( \|\Gamma_i f\|_1/\|\Gamma_i f\|_2\big)}
\]
\noindent with $C>0$ a numerical constant. It is obvious that the same proof holds at the level of the variance.\\

As exposed in \cite{CoLed}, non-product measures can also be investigated. For instance, if $\mu$ stands for the uniform probability measure on the sphere $\mathbb{S}^{n-1}$, one may consider the following fact
\[
\mathcal{E}(f,f)=\int_{\mathbb{S}^{n-1}}f(-\Delta f)d\mu=\frac{1}{2}\sum_{i,j=1}\int_{\mathbb{S}^{n-1}}(D_{i,j}f)^2d\mu
\]

\noindent where the direction $D_{ij}=x_i\partial_j-x_j\partial_i,\, i,j=1,\ldots,n$. The operators $D_{ij}$ commute in an essential way to the spherical Laplacian $\Delta=\frac{1}{2}\sum_{i,j=1}^nD_{ij}^2$ so that \eqref{eq.commutation2} holds with $\kappa=0$. However, the monotone properties needed in the proof (in order to apply Harris's negative association inequality) seems more complicated to easily characterized. \\

\section{About the upper tail}\label{5}

A similar result as Theorem \ref{thm.paouris.valettas} or \ref{thm2} has also been obtained in \cite{Dung}. Instead of convexity, the author of \cite{Dung} assumes that $f$ belongs to the set 

\[
\mathcal{F}_-=\{ f\in C^2(\R^n,\R)\quad ;\quad \text{monotone with}\quad \partial_{ij}^2f\leq 0\,\, \forall i,j=1,\ldots,n\}
\]
and obtained the following deviation inequality 

\begin{thm}[Nguyen Tien]\label{thm.dung}
Let $f\in \mathcal{F}$ be, then the following holds

\begin{equation}\label{eq.dung}
\gamma_n\bigg(f-\int_{\R^n}fd\gamma_n\geq t\bigg)\leq e^{-\frac{t^2}{{\rm Var}_{\gamma_n}(f)}},\quad t\geq 0
\end{equation}
 \end{thm}

We want to highlight the fact that the arguments used in \cite{Dung} can be easily expressed in terms of semigroup arguments. As we focus on the Gaussian case, notice that $(P_t)_{t\geq 0}$ stands for the Ornstein-Uhlenbeck semigroup. This reformulation gives shorter proof as we will show in the sequel. Unfortunately, the strategy presented below relies on exact commutation and can not be extended to the measure $\mu$. \\




Following \cite{Dung}, introduce the operator $T_g$ defined as follows

\[
T_g(y)=\int_0^\infty \E_{\gamma_n}[\nabla f(X)\cdot\nabla P_t (g)(y)]dt\quad \text{with}\quad  y\in \R^n
\]
where $f\,:\,\R^n\to \R$ is fixed, $g\,:\,\R^n\to\R$ is centered under $\gamma_n$ and $\mathcal{L}(X)=\gamma_n$.

\begin{lem}\label{lem.thm1}
With the preceding notations, for any $\theta\geq 0$, we have

\[
\E_{\gamma_n}[e^{\theta f}g]=\theta\E_{\gamma_n}[e^{\theta f}T_g]
\]
\end{lem}
\begin{proof}
Since $g$ is centered under $\gamma_n$ and by ergodicity \eqref{eq.ergodicity} of $(P_t)_{t\geq 0}$, we have

\[
\E_{\gamma_n}[e^{\theta f}(g-\E_{\gamma_n}[g])]=\E_{\gamma_n}[e^{\theta f}(P_0(g)-P_{\infty}(g))]
\] 
Thus, by the fundamental Theorem of calculus, we have 

\begin{eqnarray*}
\E_{\gamma_n}[e^{\theta f}g]&=&\E_{\gamma_n}\bigg[e^{\theta f}\bigg(-\int_0^\infty\frac{d}{dt}P_t(g) dt\bigg)\bigg]\\
&=&\int_{0}^\infty \E_{\gamma_n}\big[ e^{\theta f}\big(-LP_t(g)\big)\big]dt\quad\quad \text{by}\quad \eqref{eq.heat.equation}\\
&=&\int_0^\infty \E_{\gamma_n}[\nabla e^{\theta f}\cdot \nabla P_t(g)]dt\quad\quad \text{by}\quad \eqref{eq.ipp}\\
&=& \theta \E_{\gamma_n}[e^{ \theta f}T_g]
\end{eqnarray*}

\end{proof}
\begin{rem}
The use of the operator $T_g$ was the main idea of the article \cite{Dung}, we state it in a slightly different way which avoids a lot of calculus. For further purposes, observe that $\E_{\gamma_n}[T_g]={\rm Cov}_{\gamma_n}(f,g)$. In particular, 
\begin{equation}\label{eq.variance}
\E_{\gamma_n}[T_f]={\rm Var}_{\gamma_n}(f).
\end{equation}
\end{rem}

As in \cite{Dung}, the proof of Theorem \ref{thm.dung} relies on Lemma \ref{lem.thm1}.\\

For notational convenience, set $m=\int_{\R^n}fd\gamma_n$. Then, for any $\theta\geq 0$ define $\psi(\theta)=\E_{\gamma_n}[e^{\theta(f-m)}]$. From Lemma \ref{lem.thm1}, we have 

\[
\psi'(\theta)=\E_{\gamma_n}[e^{\theta(f-m)}(f-m)]=\theta\E_{\gamma_n}[e^{\theta(f-m)}T_f].
\]

\noindent Besides, 
\[
\theta\E_{\gamma_n}[e^{\theta(f-m)}T_f]=\theta\E_{\gamma_n}\bigg[e^{\theta(f-m)}\big(T_f-{\rm Var}_{\gamma_n}(f)\big)\bigg]+\theta{\rm Var}_{\gamma_n}(f)\E_{\gamma_n}[e^{\theta(f-m)}].
\]
 To conclude, it is enough to show that $\E_{\gamma_n}\bigg[e^{\theta(f-m)}\big(T_f-{\rm Var}_{\gamma_n}(f)\big)\bigg]\leq 0$. Indeed, if it is the case we have 

\[
\psi'(\theta)\leq \theta {\rm Var}_{\gamma_n}(f)\psi(\theta)
\]
Once integrated, this differential inequality yields 

\begin{equation}\label{eq.herbst}
\E_{\gamma_n}[e^{\theta (f-m)}]\leq e^{{\rm Var}_{\gamma_n}(f)\frac{\theta^2}{2}}\quad \text{for all}\quad \theta\geq 0.
\end{equation}
Finally, the deviation inequality from Theorem \ref{thm.dung} is obtained by classical arguments : one has to use Chernoff inequality and optimize in $\theta\geq 0$.\\

Now, let us show that $\E_{\gamma_n}\bigg[e^{\theta(f-m)}\big(T_f-{\rm Var}_{\gamma_n}(f)\big)\bigg]\leq 0$. To this task, use Lemma \ref{lem.thm1} with $g=T_f-{\rm Var}_{\gamma_n}(f)$ (which, according to \eqref{eq.variance}, is centered under $\gamma_n$) to get 

\[
\E_{\gamma_n}\bigg[e^{\theta(f-m)}\big(T_f-{\rm Var}_{\gamma_n}(f)\big)\bigg]=\theta\E_{\gamma_n}[e^{\theta(f-m)}T_{T_f}]
\]
Now, let us investigate $T_{T_f}$, thanks to the exact commutation property  \eqref{eq.ornstein.commutation}, we have, for any $y\in\R^n$ (omitted here),

\begin{eqnarray*}
T_{T_f}&=& \int_0^\infty e^{-t}\E_{\gamma_n}[\nabla f\cdot P_t(\nabla T_f)]dt\\
&=&\int_0^\infty e^{-t}\sum_{i=1}^n\E_{\gamma_n}[\partial_if P_t(\partial_i T_f)]dt
\end{eqnarray*}
Besides, for any $i=1,\ldots, n$, 

\begin{eqnarray*}
\partial_i T_f=\partial_i \int_0^\infty \sum_{j=1}^n\E_{\gamma_n}[\partial_j f\partial_jP_t(f)]dt
&= & \int_0^\infty \sum_{j=1}^n\E_{\gamma_n}[\partial_{ij}^2 f\partial_j P_t(f)+\partial_jf\partial_{ij}^2P_t(f)]dt\\
&=& \int_0^\infty \sum_{j=1}^ne^{-t}\E_{\gamma_n}[\partial_{ij}^2fP_t(\partial_j f)+e^{-t}\partial_j fP_t(\partial_{ij}^2 f)]dt\\
&\leq & 0\\
\end{eqnarray*}
since $f\in\mathcal{F}_-$. Thus, $\E_{\gamma_n}\bigg[e^{\theta(f-m)}\big(T_f-{\rm Var}_{\gamma_n}(f)\big)\bigg]\leq 0$ and the proof is complete.\\

\begin{rem}
Theorem \ref{thm.dung} implicitly uses a covariance identity (through the operator $T_f$). Similar identities have been used in \cite{Hou3} for infinitely divisible random vectors having finite exponential moments. In particular, sharp deviation inequalities were obtained. We wonder if our result can be extend to this level of generality.\\
\end{rem}

\textit{Acknowledgment : I would like to thank P. Valettas for several comments and precious remarks. I warmly thank the referee for helpful comments in improving the exposition.}

\end{document}